\documentclass{amsart}

\usepackage{mathrsfs}
\usepackage{graphicx,amssymb,xcolor,enumerate}
\usepackage{hyperref}
\usepackage{verbatim}
\hypersetup{
    bookmarks=true,         
    pdffitwindow=false,     
    pdfstartview={FitH},    
    colorlinks=false,       
}
 \newtheorem{thm}{Theorem}[section]
 
 \newtheorem{lem}[thm]{Lemma}
 
 \theoremstyle{definition}
 \newtheorem{defn}[thm]{Definition}
 \theoremstyle{remark}

 \numberwithin{equation}{section}

\usepackage{graphicx}
\usepackage{amssymb, amsmath, amsthm}
\usepackage{amsfonts}
\usepackage{amssymb}
\usepackage{float}
\usepackage{mathrsfs}
\usepackage{enumitem} 
\newcommand{\trho}{\tilde\rho}
\newcommand{\dis}{\displaystyle}

\newcommand{\divv}{\text{\rm div}}

\newcommand{\R}{\mathbb R}

\newcommand{\intox}{\int_{\R^3}}

\newcommand{\intoxt}{\int_0^t\int_{\R^3}}

\numberwithin{equation}{section}
\numberwithin{theorem}{section}

\numberwithin{figure}{section}
\begin{document}

%
%
%
%
%
%
%
%
%




\title[Global regularity for MHD equations]
 {Global regularity for the 3D compressible magnetohydrodynamics with general pressure}
 
\author{Anthony Suen} 

\address{Department of Mathematics and Information Technology\\The Education University of Hong Kong, Hong Kong}

\email{acksuen@eduhk.hk}


\date{\today}



\begin{abstract} 

We address the compressible magnetohydrodynamics (MHD) equations in $\R^3$ and establish a blow-up criterion for the local strong solutions in terms of the density only. Namely, if the density is away from vacuum ($\rho= 0$) and the concentration of mass ($\rho=\infty$), then a local strong solution can be continued globally in time. The results generalise and strengthen the previous ones in the sense that there is no magnetic field present in the criterion and the assumption on the pressure is significantly relaxed. The proof is based on some new a priori estimates for three-dimensional compressible MHD equations.

\end{abstract}

\maketitle
\section{Introduction}
In this paper, we are concerned with the compressible magnetohydrodynamics (MHD) in three space dimensions. The fluid motion is described in the following system of partial differential equations (see Cabannes \cite{cabannes} for a more comprehensive discussion on the system):
\begin{align}
 \rho_t + \divv (\rho u) &=0, \label{MHD1}\\
(\rho u^j)_t + \divv (\rho u^j u) + P(\rho)_{x_j} + ({\textstyle\frac{1}{2}}|B|^2)_{x_j}-\divv (B^j B)&= \mu \Delta u^j +
\lambda \, \divv \,u_{x_j},\label{MHD2}\\ 
B^{j}_{t} + \divv (B^j u - u^j B) &=\nu\Delta B^j,\label{MHD3}\\
\divv\,B &= 0.\label{MHD4}
\end{align}
Here $\rho$, $u=(u^1,u^2,u^3)$ and $B=(B^1,B^2,B^3)$ are functions of $x\in\R^3$ and $t\ge0$ representing density, velocity and magnetic field; $P=P(\rho)$
is the pressure; $\mu$, $\lambda$, $\nu$ are viscous constants. The system \eqref{MHD1}-\eqref{MHD4} is solved subjected to some given initial data:
\begin{align}
(\rho,u,B)(x,0)&=(\rho_0,u_0,B_0)(x).\label{1.5}
\end{align}

The well-posedness of the MHD system \eqref{MHD1}-\eqref{MHD4} have been studied by many mathematicians in decades (see for example \cite{CT10, HW10, HW08, kawashima83, LSW11, LSX16, LY11, sart09, suenhoff12, suen20a} and the references therein), and we now give a brief review on the related results. When the initial data is taken to be close to a constant state in $H^3$, Kawashima \cite{kawashima83} proved the existence of global-in-time $H^3$ solutions to the MHD system. Later, Hoff and Suen \cite{suenhoff12} generalised Kawashima's results to obtain global smooth solutions when the initial data is taken to be $H^3$ but only close to a constant state in $L^2$. The existence of global weak solutions with large initial data was proved by Hu and Wang \cite{HW10, HW08} and Sart \cite{sart09}, which are extensions of Lions-type weak solutions  \cite{lions} for the Navier-Stokes system. With initial $L^2$ data close to a constant state, Hoff and Suen \cite{suenhoff12, suen12, suen20a} generalised Hoff-type intermediate weak solutions \cite{hoff95, hoff05, hoff06, LS16, suen13b, suen14, suen16, suen20b} to obtain global solutions to the MHD system.

On the other hand, the global existence of smooth solution to \eqref{MHD1}-\eqref{MHD4} with arbitrary smooth data is still unknown, hence it is reasonable to consider the possibilities of blowing up of smooth solution. For the corresponding Navier-Stokes system, Xin \cite{xin98} proved that smooth solution will blow up in finite time in the whole space when the initial density has compact support, while Rozanova \cite{rozanova08} showed similar results for rapidly decreasing initial density. Fan-Jiang-Ou \cite{FJO10}, Sun-Wang-Zhang \cite{swz11} and Suen \cite{suen20c} established some blow-up criteria for the classical solutions to 3D compressible flows, which were further extended by Lu-Du-Yao \cite{LDY12} for MHD system. For the isothermal case when $P(\rho)=K\rho$ for some $K>0$, it was proved in \cite{suen13a, suen15} that without vacuum in the initial density, when the density and magnetic field are essential bounded, the smooth solutions to \eqref{MHD1}-\eqref{MHD4} can be extended globally in time.

The main goal of the present paper is to generalise and strengthen the corresponding results of \cite{suen13a, suen15}. The main novelties of this current work can be summarised as follows:
\begin{itemize}
\item We introduce some new type of estimates on functionals which are used for decoupling the velocity and magnetic field;
\item We obtain a blow-up criterion for \eqref{MHD1}-\eqref{MHD4} for general pressure term $P(\rho)$ which is {\it not} restricted to the isothermal case;
\item We assure that the blow-up criterion dependsds {\it only} on density, which is an improvement for the results obtained from \cite{suen13a, suen15} in which the magnetic field was present in the criterion;
\item We do not impose any extra compatibility condition on the initial data, which is required in the work \cite{swz11, HLX11, suen20c}.
\end{itemize} 

We give a brief description on the analysis applied in this work. To extract the ``hidden regularity'' from the velocity $u$ and magnetic field $B$, we introduce an important canonical variable associated with the system \eqref{MHD1}-\eqref{MHD4}, which is known as the {\it effective viscous flux}. To see how it works, by the Helmholtz decomposition of the mechanical forces, we can rewrite the momentum equation \eqref{MHD2} as follows (summation over $k$ is understood):
\begin{equation}\label{derivation for F}
\rho\dot u^j +( {\textstyle\frac{1}{2}}|B|^2)_{x_j}-\divv(B^jB)=F_{x_j}+\mu\omega^{j,k}_{x_k},
\end{equation}
where $\dot u^j=u^j_t+u\cdot u^j$ is the material derivative on $u^j$ and the effective viscous flux $F$ is defined by
\begin{equation}\label{definition of F}
F=(\mu+\lambda){\rm div}(u) - P(\rho) + P(\tilde{\rho}).
\end{equation}
Differentiating \eqref{derivation for F}, we obtain the following Poisson equation
\begin{equation}\label{poisson in F}
\Delta F = \divv(g),
\end{equation}
where $g^j=\rho\dot{u}^j+(\frac{1}{2}|B|^2)_{x_j}-\divv(B^jB)$. The Poisson equation \eqref{poisson in F} can be viewed as the analog for compressible MHD of the well-known elliptic equation for pressure in incompressible flow. By exploring the Rankine-Hugoniot condition (see \cite{suenhoff12}), one can deduce that the effective viscous flux $F$ is relatively more regular than $\divv(u)$ or $P(\rho)$, and it turns out to be crucial for the overall analysis in the following ways:
\begin{enumerate}

\item[(i)] The equation \eqref{derivation for F} allows us to decompose the acceleration density $\rho \dot u$ as the sum of the gradient of the scalar $F$ and the divergence-free vector field $\omega^{\cdot,k}_{x_k}$ (we ignore those lower-order terms involving $B$). The skew-symmetry of $\omega$ insures that these two vector fields are orthogonal in $L^2(\R^3)$, so that $L^2$-bounds for the terms on the left side of \eqref{derivation for F} immediately give $L^2$ bounds for the gradients of both $F$ and $\omega$. These in turn will be used for controlling $\nabla u$ in $L^4$ when $u(\cdot,t)\notin H^2$, which are crucial for estimating different functionals in $u$ and $B$; see section~\ref{A priori estimates section} and the proof of Theorem~\ref{a priori bound thm}.

\item[(ii)] One of the key step in obtaining higher order estimates on the solutions is to bound the time integral of $\|\nabla u(\cdot,t)\|_{L^\infty}$. The key observation is to decompose $u$ as $u=u_{F}+u_{P}$, where $u_{F}$, $u_{P}$ satisfy
\begin{align*}
\left\{
 \begin{array}{lr}
(\mu+\lambda)\Delta u_{F}^{j}=F_{x_j} +(\mu+\lambda)\omega^{j,k}_{x_k}\\
(\mu+\lambda)\Delta u_P^{j}=(P-P(\tilde{\rho}))_{x_j}.\\
\end{array}
\right.
\end{align*}
Using the {\it a priori} bounds on the effective viscous flux $F$, the time integral of $\|\nabla u_F(\cdot,t)\|_{L^\infty}$ can be estimated in terms of $F$, which can be further estimated in terms of some {\it a priori} bounds on $\dot{u}$ and $B$. On the other hand, to control $\int_{0}^{t}||\nabla u_{P}(\cdot,s)||_{\infty}ds$, by applying the analysis on Newtonian potentials given in \cite{BC94}, one can show that if $\Gamma$ is the fundamental solution for the Laplace operator on $\R^3$, then $u^j_P(\cdot,t)=(\mu+\lambda)^{-1}\Gamma_{x_j}*(P(\rho(\cdot,t))-\tilde P)$ is {\it log-Lipschitz} provided that $P(\rho(\cdot,t))\in L^\infty$ holds. This is sufficient to guarantee that the integral curve $x(\cdot,t)$ of $u=u_F+u_P$ is H\"{o}lder continuous under the assumption that $u_F$ has enough regularity as claimed. If we assume that the initial density is H\"{o}lder continuous, then using the mass equation \eqref{MHD1}, it implies that the density is also H\"{o}lder continuous for positive time. Hence with such improved regularity on the density, it allows us to obtain the desired bound on $u_{P}$; see section~\ref{higher order section} and the proof of Theorem~\ref{higher order bounds thm}. Such method was also exploited in \cite{suen20a} for proving global-in-time existence and uniqueness of weak solutions to \eqref{MHD1}-\eqref{MHD4}. 
\end{enumerate}

We now give a precise formulation of our results. First concerning the assumptions on the parameters, we have:
\begin{equation}\label{def of P}
\mbox{$P(\cdot)$ is a $C^2$-function in $\rho$ with $P'(\rho)>0$ for all $\rho>0$;}
\end{equation}
\begin{equation}\label{condition on vis}
\mbox{$\mu$, $\lambda$, $\nu>0$ with $\mu>4\lambda$.}
\end{equation}

Given $\tilde\rho>0$, for the initial data, we assume that
\begin{align}\label{1.8}
\mbox{$\rho_0-\tilde\rho,u_0,B_0\in H^3(\R^3)$ with $\inf(\rho_0)>0$ and $\divv(B_0)=0$.}
\end{align}
 
The following is the main result of this paper:

\begin{thm}\label{main thm}
Assume that the system parameters satisfy \eqref{def of P}-\eqref{condition on vis}. Given $\tilde\rho>0$, suppose $(\rho_0-\tilde\rho,u_0,B_0)$ satisfies \eqref{1.8}. Assume that $(\rho-\tilde\rho,u,B)$ is the smooth solution local-in-time solution to \eqref{MHD1}-\eqref{MHD4} as defined on $\R^3\times[0,T]$, and let $T^*\ge T$ be the maximal existence time of the solution. If $T^* <\infty$, then we have
\begin{align*}
\lim_{t\rightarrow T^*}(||\rho||_{L^\infty((0,t)\times\R^3)}+||\rho^{-1}||_{L^\infty((0,t)\times\R^3)})=+\infty.
\end{align*}
\end{thm}

The rest of the paper is organised as follows. In section~\ref{prelim section}, we recall some known facts and inequalities which will be useful for later analysis. In section~\ref{A priori estimates section}, we begin the proof of Theorem~\ref{main thm} with a number of {\em a priori} bounds for smooth solutions, which are summarised in Theorem~\ref{a priori bound thm}. Finally in section~\ref{higher order section}, we complete the proof of Theorem~\ref{main thm} via a contradiction argument by deriving higher order $H^3$-bounds for smooth solutions. 

\section{Preliminaries}\label{prelim section}

In this section, we recall some known facts and useful inequalities for our analysis. We first state a local existence theorem for \eqref{MHD1}-\eqref{MHD4} proved by Kawashima \cite[pg.~34--35 and pg.~52--53]{kawashima83}:

\begin{thm}\label{Kawashima thm}
Assume that the system parameters satisfy \eqref{def of P}-\eqref{condition on vis}. Then given $\tilde\rho>0$ and $\tilde C>0$, there is a positive time $T$ depending on $\tilde\rho$, $\tilde C$ and the parameters $\varepsilon,\lambda,\nu,P$ such that if the initial data $(\rho_0,u_0,B_0)$ is given satisfying \eqref{1.8} and 
\begin{align*}
C_0<\tilde C,
\end{align*}
then there is a solution $(\rho-\tilde\rho,u,B)$ to \eqref{MHD1}-\eqref{MHD4} defined on $\R^3\times[0,T]$ satisfying
\begin{equation}\label{1.1a}
\rho - \tilde{\rho}\in C([0,T];H^{3}(\R^3))\cap C^1 ([0,T];H^{2}(\R^3))
\end{equation}
and
\begin{equation}\label{1.1b}
u,B\in C([0,T];H^{3}(\R^3))\cap C^1 ([0,T];H^1 (\R^3))\cap L^2([0,T];H^4 (\R^3)).
\end{equation}
\end{thm}

We make use of the following standard facts (see Ziemer \cite[Theorem~2.1.4, Remark~2.4.3, and Theorem~2.4.4]{ziemer} for example). First, given $r\in[2,6]$ there is a constant $C(r)$ such that for $w\in H^1 (\R^3)$,
\begin{equation}\label{Lr bound general}
\|w\|_{L^r(\R^3)} \le C(r) \left(\|w\|_{L^2(\R^3)}^{(6-r)/2r}\|\nabla w\|_{L^2(\R^3)}^{(3r-6)/2r}\right).
\end{equation}
For any $r\in (1,\infty)$ there is a constant $C(r)$ such that for $w\in W^{1,r}(\R^3)$,
\begin{equation}\label{L infty bound general}
\|w\|_{L^\infty (\R^3)} \le C(r) \|w\|_{W^{1,r}(\R^3)}.
\end{equation}

\section{A priori estimates}\label{A priori estimates section}

In this section we derive {\em a priori} estimates for the local solution $(\rho-\tilde\rho,u,B)$ on $[0,t]$ with $t\in[0,T^*)$ as described by Theorem~\ref{main thm}. Here $T^*$ is the maximal time of existence which is defined in the following sense:

\begin{defn}
We call $T^*\in(0, \infty)$ to be the maximal time of existence of a smooth solution $(\rho-\tilde\rho,u,B)$ to \eqref{MHD1}-\eqref{MHD4} if for any $t\in[0,T^*)$, $(\rho-\tilde\rho,u,B)$ solves \eqref{MHD1}-\eqref{MHD4} in $[0, t]\times\R^3$ and satisfies \eqref{1.1a}-\eqref{1.1b}; moreover, the conditions \eqref{1.1a}-\eqref{1.1b} fail to hold when $t=T^*$.
\end{defn}

We will prove Theorem~\ref{main thm} using a contradiction argument. Therefore, for the sake of contradiction, we assume that
\begin{align}\label{pointwise bound on rho finite time}
||\rho||_{L^\infty((0,T^*)\times\R^3)}+||\rho^{-1}||_{L^\infty((0,T^*)\times\R^3)}\le M_0.
\end{align}
To facilitate our exposition, we first define some auxiliary functionals for $t\in[0,T^*)$:
\begin{align*}
\Phi_1(t)=\sup_{0\le s \le t}\int_{\R^3}(|\nabla u|^2+|\nabla B|^2) +\int_0^t\int_{\R^3}(|\dot u|^2+|B_t|^2),
\end{align*}
\begin{align*}
\Phi_2(t)=\sup_{0\le s \le t}\int_{\R^3}(|\dot u|^2+|B_t|^2) +\int_0^t\int_{\R^3}(|\nabla\dot u|^2+|\nabla B_t|^2),
\end{align*}
\begin{align*}
\Phi_3(t)=\int_0^t\int_{\R^3}(|\nabla u|^4+|\nabla B|^4),\qquad \Phi_4(t)=\intoxt (|\nabla u|^3+|\nabla B|^3).
\end{align*}
\begin{align*}
\Phi_5(t)=\sup_{0\le s\le t}\intox (|\nabla B|^2|B|^2+|\nabla u|^2|B|^2+|\nabla B|^2|u|^2).
\end{align*}
The following is the main theorem of this section:
\begin{thm}\label{a priori bound thm}
Assume that the hypotheses and notations in Theorem~\ref{main thm} are in force. Given $M_0>0$ and $\tilde\rho>0$, assume further that $(\rho-\tilde\rho,u,B)$ satisfies \eqref{pointwise bound on rho finite time}. Then for each $t\in[0,T^*)$, there exists a positive constant $C$ which depends on $M_0$, $t$, $\tilde\rho$ and the system parameters $P$, $\mu$, $\lambda$, $\nu$ such that
\begin{align}\label{bounds on phi 1 and 2}
\Phi_1(t)+\Phi_2(t)\le C.
\end{align}
\end{thm}

We prove Theorem~\ref{a priori bound thm} in a sequence of lemmas. Throughout this section, for $t\in[0,T^*)$, $C$ always denotes a generic constant which depends only on $\mu$, $\lambda$, $a$, $\nu$, $\trho$, $M_0$, $t$ and the initial data. For simplicity, we drop the symbols $dx$, $ds$ or $dxds$ from the integrals.

We first give the following estimates on the effective vicious flux $F$ which is defined in \eqref{definition of F}.

\begin{lem}\label{estimate on effective viscous flux}
Assume that $\rho$ satisfies \eqref{pointwise bound on rho finite time}. For each $p>1$, there is a constant $C>0$ such that for all $t>0$, we have
\begin{equation}\label{bound on F}
\|F(\cdot,t)\|_{L^p}\le\Big[\|\nabla u(\cdot,t)\|_{L^p}+\|(\rho-\tilde\rho)(\cdot,t)\|_{L^p}\Big],
\end{equation}
and
\begin{equation}\label{bound on nabla F}
\|\nabla F(\cdot,t)\|_{L^p}\le\Big[\|\dot{u}(\cdot,t)\|_{L^p}+\|B\nabla B(\cdot,t)\|_{L^p}\Big].
\end{equation}
\end{lem}
\begin{proof}
The assertion \eqref{bound on F} follows immediately from the definition of $F$, and the proof of \eqref{bound on nabla F} relies on the Poisson equation \eqref{poisson in F} and the Marcinkiewicz multiplier theorem (refer to Stein \cite{stein70}, pg. 96). 
\end{proof}

Using the estimates \eqref{bound on F}-\eqref{bound on nabla F} on $F$, we have the following estimates on $\nabla u$ and $\nabla\omega$:

\begin{lem}\label{estimate on nabla u lemma}
Assume that $\rho$ satisfies \eqref{pointwise bound on rho finite time}. For each $p>1$, there is a constant $C>0$ depends on $p$ such that for all $t>0$, we have
\begin{align}\label{bound on nabla u}
\|\nabla u(\cdot,t)\|_{L^p}\le C\Big[\|F(\cdot,t)\|_{L^p}+\|\omega(\cdot,t)\|_{L^p}+\|(P-\tilde P)(\cdot,t)\|_{L^p}\Big],
\end{align}
\begin{align}\label{bound on omega}
\|\nabla\omega(\cdot,t)\|_{L^p}\le C\Big[\|\dot{u}(\cdot,t)\|_{L^p}+\|B\nabla B(\cdot,t)\|_{L^p}\Big].
\end{align}
\end{lem}

\begin{proof} By the definition \eqref{definition of F} of $F$,
\begin{equation*}
(\mu+\lambda)\Delta u^j=F_{x_j}+(\mu+\lambda)\omega^{j,k}_{x_k}+(P-\tilde{P})_{x_j}.
\end{equation*}
Hence by differentiating and taking the Fourier transform on the above equation, we can apply Marcinkiewicz multiplier theorem and \eqref{bound on nabla u} follows.

For the case of $\nabla\omega$, by direct computation, we have
\begin{equation*}
\mu\Delta\omega=(\rho\dot{u}^j)_{x_k}-(\rho\dot{u}^k)_{x_j}-(\nabla B^{j}\cdot B)_{x_k}+(\nabla B^{k}\cdot B)_{x_j},
\end{equation*}
and using the same argument as for $\nabla u$, \eqref{bound on omega} follows immediately.
\end{proof}

We are now ready for giving the estimates for proving Theorem~\ref{a priori bound thm}. We begin with the following $L^2$ estimates on $(\rho,u,B)$ for all $t\in[0,T^*)$:

\begin{lem}\label{L2 estimate lem}
Assume that $\rho$ satisfies \eqref{pointwise bound on rho finite time}. For $t\in[0,T^*)$, we have
\begin{align}\label{L2 bound}
\sup_{0\le s\le t}\intox(|\rho-\tilde\rho|^2+\rho|u|^2+|B|^2)+\intoxt(|\nabla u|^2+|\nabla B|^2)\le C.
\end{align}
\end{lem}
\begin{proof}
The bound \eqref{L2 bound} follows from the standard energy balance equation and we refer to \cite{hoff95} or \cite{suen20b} for related discussions.
\end{proof}

Next we derive the following $L^6$ bounds for $u$ and $B$. Such bounds are crucial for obtaining higher order estimates on $u$ and $B$.
\begin{lem}\label{L6 bound lem}
Assume that the hypotheses and notations of  {\rm Theorem~\ref{a priori bound thm}} are in force. Then for any $0\le t< T^*$,
\begin{align}\label{L6 bound}
\sup_{0\le s\le t}\intox(|u|^6+|B|^6)\le C.
\end{align}
\end{lem}
\begin{proof}
We follow the computations given in \cite{suenhoff12, suen20a} and obtain, for $t\in[0,T^*)$,
\begin{align*}
&\intox(|u|^6+|B|^6)\Big|_{s=0}^t+6\int_0^t\intox(\mu|u|^2|\nabla u|^2+\nu|B|^2|\nabla B|^2)\\
&\qquad+(-24\lambda+6\mu)\int_0^t\intox|u|^2|\nabla(|u|^2)|^2+6\nu\int_0^t\intox|B|^2|\nabla)|B|^2)|^2\\
&\le C\Big[\int_0^t\intox|\rho-\tilde\rho||\divv(|u|^4u)|+\int_0^t\intox|u|^4|u||\divv(BB^T)|\Big]\\
&\qquad+C\Big[\int_0^t\intox|u|^4|u||\nabla(\frac{1}{2}|B|^2)|+\int_0^t\intox|B|^4|B\cdot\divv(Bu^T-uB^T)|\Big].
\end{align*}
By the assumption \eqref{condition on vis}, the term involving $(-24\lambda+6\mu)$ is positive. The rest of the analysis follows by a Gr\"{o}nwall-type argument and we omit the details here. 
\end{proof}
With the help of \eqref{L2 bound} and \eqref{L6 bound}, we can start estimating the functionals $\Phi_1$ and $\Phi_2$ appeared in \eqref{bounds on phi 1 and 2}. We first consider $\Phi_1$ and $\Phi_4$:
\begin{lem}\label{bound on A1 and Phi1 lem}
Assume that the hypotheses and notations of  {\rm Theorem~\ref{a priori bound thm}} are in force. Then for any $0\le t< T^*$,
\begin{align}\label{bound on A1 and Phi1}
\Phi_1(t)+\Phi_4(t)\le C.
\end{align}
\end{lem}
\begin{proof}
We multiply \eqref{MHD2} by $\dot{u}^j$, sum over $j$ and integrate to get
\begin{align}\label{2.17}
&\int_{\R^3}|\nabla u|^2+\int_0^t\int_{\R^3}\rho|\dot u|^2\notag\\
&\qquad\le C+\left|\int_0^t\int_{\R^3}[\dot u\cdot\nabla(\frac{1}{2}|B|^2)-\dot{u}\cdot\divv(BB^T)]\right|+\int_0^t\int_{\R^3}|\nabla u|^3.
\end{align}
Niext we multiply \eqref{MHD3} by $B_t$ and integrate,
\begin{align}\label{2.18}
\int_{\R^3}|\nabla B|^2+\int_0^t\int_{\R^3}|B_t|^2
\le C+\left|\int_0^t\int_{\R^3}B_t \cdot\divv(uB^T-u^TB)\right|.
\end{align}
Adding \eqref{2.17} and \eqref{2.18}, we obtain
\begin{align}\label{2.19}
&\int_{\R^3}(|\nabla u|^2+|\nabla B|^2)+\int_0^t\int_{\R^3}(\rho|\dot u|^2+|B_t|^2)\notag\\
&\le C+C\Phi_4+C\int_0^t\int_{\R^3}(|\nabla B|^2|B|^2+|\nabla u|^2|B|^2+|\nabla B|^2|u|^2).
\end{align}
To bound the last inequality on the right side of \eqref{2.19}, we bound the term $\intoxt|\nabla B|^2|B|^2d$ as an example. Using \eqref{Lr bound general} and the bound \eqref{L6 bound}, we arrive at
\begin{align*}
\intoxt|\nabla B|^2|B|^2&\le\Big(\intoxt|\nabla B|^3\Big)^\frac{2}{3}\Big(\intoxt|B|^6\Big)^\frac{1}{3}\\
&\le C\Big(\int_0^t\Big(\intox|\Delta B|^2\Big)^\frac{3}{4}\Big(\intox|\nabla B|^2\Big)^\frac{3}{4}\Big)^\frac{2}{3}\\
&\le C\Big(\int_0^t\Big(\intox(|B_t|^2+|\nabla B|^2|u|^2+|\nabla u|^2|B|^2)\Big)^\frac{3}{4}\Big(\intox|\nabla B|^2\Big)^\frac{3}{4}\Big)^\frac{2}{3}\\
&\le C\Big(\sup_{0\le s\le t}\intox|\nabla B|^2\Big)^\frac{1}{3}\Big(\intoxt|B_t|^2\Big)^\frac{1}{2}\Big(\intoxt|\nabla B|^2\Big)^\frac{1}{6}\\
&\qquad+C\Big(\int_0^t\Big(\intox|\nabla B|^3\Big)^\frac{1}{2}\Big(\intox|u|^6\Big)^\frac{1}{4}\Big(\intox|\nabla B|^2\Big)^\frac{3}{4}\Big)^\frac{2}{3}\\
&\qquad+C\Big(\int_0^t\Big(\intox|\nabla u|^3\Big)^\frac{1}{2}\Big(\intox|B|^6\Big)^\frac{1}{4}\Big(\intox|\nabla B|^2\Big)^\frac{3}{4}\Big)^\frac{2}{3}\\
&\le C\Phi_1^\frac{5}{6}+C\Phi_1^\frac{1}{6}\Phi_4^\frac{1}{3}.
\end{align*}
The estimates on $\intoxt|\nabla u|^2|B|^2$ and $\intoxt|\nabla B|^2|u|^2$ are just similar, and we deduce that
\begin{align}\label{bound on Phi1 by A1}
\Phi_1\le C\Phi_1^\frac{5}{6}+C\Phi_1^\frac{1}{6}\Phi_4^\frac{1}{3}.
\end{align}
It remains to estimate the functional $\Phi_4$. Using \eqref{Lr bound general}, we can estimate the integral of $|\nabla B|^3$ as follows.
\begin{align*}
\intoxt|\nabla B|^3&\le C\int_0^t\Big(\intox|\Delta B|^2\Big)^\frac{3}{4}\Big(\intox|\nabla B|^2\Big)^\frac{3}{4}\\ 
&\le C\int_0^t\Big(\intox(|B_t|^2+|\nabla B|^2|u|^2+|\nabla u|^2|B|^2)\Big)^\frac{3}{4}\Big(\intox|\nabla B|^2\Big)^\frac{3}{4}\\
&\le C\Big(\sup_{0\le s\le t}\intox|\nabla B|^2\Big)^\frac{1}{2}\Big(\intoxt|B_t|^2\Big)^\frac{3}{4}\Big(\intoxt |\nabla B|^2\Big)^\frac{1}{4}\\
&\qquad+C\int_0^t\Big(\intox|\nabla B|^3\Big)^\frac{1}{2}\Big(\intox|u|^6\Big)^\frac{1}{4}\Big(\intox|\nabla B|^2\Big)^\frac{3}{4}\\
&\qquad+C\int_0^t\Big(\intox|\nabla u|^3\Big)^\frac{1}{2}\Big(\intox|B|^6\Big)^\frac{1}{4}\Big(\intox|\nabla B|^2\Big)^\frac{3}{4}\\
&\le C\Phi_1^\frac{5}{4}+C\Phi_1^\frac{1}{4}\Phi_4^\frac{1}{2}.
\end{align*}
On the other hand, using \eqref{Lr bound general}, \eqref{bound on nabla u}, \eqref{bound on omega} and \eqref{L2 bound}, the integral of $|\nabla u|^3$ can be estimated as follows.
\begin{align*}
&\intoxt |\nabla u|^3\\
&\le C\intoxt (|F|^3+|\omega|^3+|P-\tilde P|^3)\\
&\le C\int_0^t \Big(\intox|F|^2\Big)^\frac{3}{4}\Big(\intox|\nabla F|^2\Big)^\frac{3}{4}+C\int_0^t \Big(\intox|\omega|^2\Big)^\frac{3}{4}\Big(\intox|\nabla \omega|^2\Big)^\frac{3}{4}+C\\
&\le C\Big(\sup_{0\le s\le t}\intox|\nabla u|^2\Big)^\frac{1}{2}\Big(\intoxt |\dot{u}|^2\Big)^\frac{3}{4}+C\Big(\intoxt |\dot{u}|^2\Big)^\frac{3}{4}\\
&\qquad+C\Big(\sup_{0\le s\le t}\intox|\nabla u|^2\Big)^\frac{1}{4}\Big(\intoxt |\nabla B|^3\Big)^\frac{1}{2}\\
&\qquad+C\Big(\intoxt |\nabla B|^3\Big)^\frac{1}{2}+C\\
&\le C\Phi_1^\frac{5}{4}+C\Phi_1^\frac{3}{4}+C\Phi_1^\frac{1}{4}\Phi_4^\frac{1}{2}+C\Phi_4^\frac{1}{2}+C.
\end{align*}
Therefore, we obtain the following bound on $\Phi_4$ in terms of $\Phi_1$:
\begin{align}\label{bound on A1 by Phi1}
\Phi_4\le C\Phi_1^\frac{5}{4}+C\Phi_1^\frac{3}{4}+C\Phi_1^\frac{1}{2}+C
\end{align}
Combining \eqref{bound on Phi1 by A1} with \eqref{bound on A1 by Phi1}, the result \eqref{bound on A1 and Phi1} follows.
\end{proof}

Next, we derive the following estimates for $\Phi_2$ in terms of $\Phi_3$ and $\Phi_5$:

\begin{lem}\label{bound on Phi2 Phi3 A2 lem} 
Assume that the hypotheses and notations of  {\rm Theorem~\ref{a priori bound thm}} are in force. Then for any $0\le t<T^*$,
\begin{align}\label{bound on Phi2 Phi3 A2}
\Phi_2(t)+\Phi_3(t)+\Phi_5(t)\le C.
\end{align}
\end{lem}
\begin{proof}
Taking the convective derivative in the momentum equation \eqref{MHD2}, multiplying it by $\dot{u}^j$, summing over $j$ and integrating, we obtain
\begin{align}\label{H2 bound step 0}
&\sup_{0\le s\le t}\int_{\R^3}|\dot{u}|^2+\int_0^t\int_{\R^3}|\nabla\dot u|^2\notag\\
&\qquad\qquad\le C+C\Phi_3+C\int_0^t\int_{\R^3}|B|^2(|B_t|^2+|u|^2|\nabla B|^2).
\end{align}
Next we differentiate the magnetic field equation \eqref{MHD3} with respect to $t$, multiply by $B_t$ and integrate,
\begin{align*}
&\left.\frac{1}{2}\int_{\R^3}|B_t|^2\right|_0^t+ \nu\int_{0}^{t}\int_{\R^3}|\nabla B_t|^2 \\
&\qquad\qquad\qquad=-\int_{0}^{t}\int_{\R^3} B_t\cdot[\divv(B u^{T}-u B^{T})]_{t} .
\end{align*}
Adding the above to \eqref{H2 bound step 0} and absorbing terms,
\begin{align}\label{H2 bound step 1}
&\Phi_2\le C\left[\Phi_3+\int_0^t\int_{\R^3}|B|^2|u|^2(|\nabla u|^2+|\nabla B|^2)+1\right]\notag\\
&\qquad\qquad\qquad\qquad+C\int_0^t\int_{\R^3}(|B|^2|B_t|^2+|B|^2|\dot u|^2+|B_t|^2|u|^2).
\end{align}
We first consider the last integral on the right side of \eqref{H2 bound step 1}. To bound the term $\intoxt |B|^2|B_t|^2$, using the bound \eqref{bound on A1 and Phi1}, we have
\begin{align*}
&\intoxt |B|^2|B_t|^2\\
&\le\Big(\intoxt|B|^6\Big)^\frac{1}{3}\Big(\intoxt |B_t|^3\Big)^\frac{2}{3}\\
&\le C\Big(\sup_{0\le s\le t}\intox|B_t|^2\Big)^\frac{1}{3}\Big(\intoxt |B_t|^2\Big)^\frac{1}{6}\Big(\intoxt |\nabla B_t|^2\Big)^\frac{1}{2}\\
&\le C\Phi_1^\frac{1}{6}\Phi_2^\frac{5}{6}\le C\Phi_2^\frac{5}{6},
\end{align*}
and the terms $\intoxt |u|^2|B_t|^2$ and $\intoxt |B|^2|\dot{u}|^2$ can be treated in a similar way. To bound the third integral on the right side of \eqref{H2 bound step 1}, we have
\begin{align*}
&\intoxt |B|^2|u|^2(|\nabla u|^2+|\nabla B|^2)\\
&\le \intoxt (|\nabla u|^4+|\nabla B|^4)+\intoxt (|B|^8+|u|^8)\\
&\le \Phi_3+\intoxt (|B|^8+|u|^8).
\end{align*}
Using the bounds \eqref{Lr bound general} and  \eqref{L infty bound general}, the term $\intoxt |u|^8$ can be estimated as follows.
\begin{align*}
&\intoxt |u|^8\\
&\le\Big(\sup_{0\le s\le t}\intox|u|^4\Big)\Big(\int_0^t\|u|^4_{L^\infty}\Big)\\
&\le C\Big(\sup_{0\le s\le t}\intox|u|^2\Big)^\frac{1}{2}\Big(\sup_{0\le s\le t}\intox|u|^6\Big)^\frac{1}{2}\Big(\int_0^t\intox(|u|^4+|\nabla u|^4)\Big)\\
&\le C\Big[\int_0^t\Big(\intox|u|^2\Big)^\frac{1}{2}\Big(\intox|\nabla u|^2\Big)^\frac{3}{2}+\Phi_3\Big]\\
&\le C\Big(\Phi_3+1\Big).
\end{align*}
The term $\intoxt |B|^8$ can be estimated in a similar way to get
\begin{align*}
\intoxt |B|^2|u|^2(|\nabla u|^2+|\nabla B|^2)\le C\Big(\Phi_3+1\Big),
\end{align*}
and we obtain from \eqref{H2 bound step 1} that
\begin{align}\label{H2 bound step 2}
\Phi_2\le C\left(\Phi_3+1\right)+ C\Phi_2^\frac{5}{6}.
\end{align}
It remains to estimate the functionals $\Phi_3$ and $\Phi_5$. Using \eqref{bound on nabla u} and \eqref{L2 bound}, we have
\begin{align}\label{bound on L4 of nabla u}
\intoxt|\nabla u|^4\le C\intoxt(|F|^4+|\omega|^4)+C.
\end{align}
Using \eqref{bound on A1 and Phi1}, the integral on $|F|^4$ can be bounded by
\begin{align*}
\intoxt|F|^4&\le C\Big(\sup_{0\le s\le t}\intox|F|^2\Big)^\frac{1}{2}\Big(\sup_{0\le s\le t}\intox|\nabla F|^2\Big)^\frac{1}{2}\Big(\intoxt|\nabla F|^2\Big)\notag\\
&\le C\Big(\sup_{0\le s\le t}\intox(|\dot{u}|^2+|\nabla B|^2|B|^2\Big)^\frac{1}{2}\Big(\intoxt(|\dot{u}|^2+|\nabla B|^2|B|^2)\Big)^\frac{1}{2}\notag\\
&\le C(\Phi_2+\Phi_5)^\frac{1}{2},
\end{align*}
and the estimates on $\omega$ is just similar. For $\intoxt |\nabla B|^4$, we estimate it as follows.
\begin{align*}
\intoxt |\nabla B|^4&\le C\int_0^t\Big(\intox|\nabla B|^2\Big)^\frac{1}{2}\Big(\intox|\Delta B|^2\Big)^\frac{3}{2}\\
&\le C\Big(\sup_{0\le s\le t}\intox|\Delta B|^2\Big)^\frac{1}{2}\Big(\sup_{0\le s\le t}\intox|\nabla B|^2\Big)^\frac{1}{2}\Big(\intoxt |\Delta B|^2\Big)\\
&\le C\Big(\Phi_2+\Phi_5\Big)^\frac{1}{2}.
\end{align*}
Hence we have
\begin{align}\label{bound on Phi3 by A2}
\Phi_3\le C(\Phi_2+\Phi_5)^\frac{1}{2}+C.
\end{align}
To bound $\Phi_5$, for $r\in (3,6)$, we apply \eqref{L infty bound general} to obtain
\begin{align*}
&\intox|\nabla u|^2|B|^2\\
&\le C\|B(\cdot,t)\|^2_{L^\infty}\Big(\intox|\nabla u|^2\Big)\\
&\le C\Big(\intox|B|^r+\intox|\nabla B|^r\Big)^\frac{2}{r}\Big(\intox|\nabla u|^2\Big).
\end{align*}
The term $\|B\|_{L^r}$ can be bounded by the $L^2-L^6$ interpolation on $B$. Hence using \eqref{Lr bound general} and the bound \eqref{bound on A1 and Phi1}, we obtain
\begin{align*}
\intox|\nabla u|^2|B|^2&\le C+C\Big(\intox|\nabla u|^2\Big)\Big(\intox|\nabla B|^2\Big)^\frac{6-r}{2r}\Big(\intox|\Delta B|^2\Big)^\frac{3r-6}{2r}\\
&\le C+C\Big(\intox|\nabla u|^2\Big)\Big(\intox|\nabla B|^2\Big)^{1-s'}\Big(\intox|\Delta B|^2\Big)^{s'}\\
&\le C+C\Big(\intox|B_t|^2+\intox(|\nabla B|^2|u|^2+|\nabla u|^2|B|^2)\Big)^{s'}\\
&\le C+C\Big(\Phi_2+\Phi_5\Big)^{s'},
\end{align*}
where $s':=\frac{3r-6}{2r}$ and $s'\in(0,1)$. By similar method, we obtain
\begin{align*}
\intox|\nabla B|^2|B|^2&\le \Big(\intox|\nabla B|^2\Big)\|B(\cdot,t)\|^2_{L^\infty}\\
&\le C+C\Big(\Phi_2+\Phi_5\Big)^{s'}.
\end{align*}
For the term $\intox|\nabla B|^2|u|^2$, we have
\begin{align*}
\intox|\nabla B|^2|u|^2&\le \Big(\intox|\nabla B|^2\Big)\|u(\cdot,t)\|^2_{L^\infty}\\
&\le C\Big(\intox|\nabla B|^2\Big)\Big(\intox|u|^r+\intox|\nabla u|^r\Big)^\frac{2}{r}.
\end{align*}
Similar to the case for $B$, the term $\|u\|_{L^r}$ can be bounded by the $L^2-L^6$ interpolation on $u$. For the term $\|\nabla u\|_{L^r}$, we can apply \eqref{bound on nabla u} with the bounds \eqref{bound on F}-\eqref{bound on nabla F} to get
\begin{align*}
\Big(\intox|\nabla u|^r\Big)^\frac{2}{r}&\le C\Big(\intox(|F|^r+|\omega|^r+|\rho-\tilde\rho|^r)\Big)^\frac{2}{r}\\
&\le C\Big(\intox|F|^2\Big)^{1-s'}\Big(\intox|\nabla F|^2\Big)^{s'}\\
&\qquad+C\Big(\intox|\omega|^2\Big)^{1-s'}\Big(\intox|\nabla \omega|^2\Big)^{s'}+C\Big(\intox|\rho-\tilde\rho|^r\Big)^\frac{2}{r}\\
&\le C+C\Big(\Phi_2+\Phi_5\Big)^{s'}.
\end{align*}
Therefore, we obtain the following bound on $\Phi_5$:
\begin{align}\label{bound on A2 by Phi2}
\Phi_5\le C+C\Big(\Phi_2+\Phi_5\Big)^{s'}.
\end{align}
We combine \eqref{H2 bound step 2}, \eqref{bound on Phi3 by A2} and \eqref{bound on A2 by Phi2} to conclude that
\begin{align*}
\Phi_2\le C,
\end{align*}
which can be further applied on \eqref{bound on Phi3 by A2} and \eqref{bound on A2 by Phi2} to give
\begin{align*}
\Phi_3+\Phi_5\le C.
\end{align*}
Hence the result \eqref{bound on Phi2 Phi3 A2} follows.
\end{proof}
\begin{proof}[Proof of Theorem~\ref{a priori bound thm}]
Using the results obtained from Lemma~\ref{bound on A1 and Phi1 lem} and Lemma~\ref{bound on Phi2 Phi3 A2 lem}, the bound \eqref{bounds on phi 1 and 2} follows immediately from the estimates \eqref{bound on A1 and Phi1} and \eqref{bound on Phi2 Phi3 A2}.
\end{proof}

\section{Higher Order Estimates and proof of Theorem~\ref{main thm}}\label{higher order section}

In this section we continue to obtain higher order estimates on the smooth local solution $(\rho-\tilde\rho,u,B)$ as described in section~\ref{A priori estimates section}. Together with Theorem~\ref{a priori bound thm}, we show that, under the assumption \eqref{pointwise bound on rho finite time}, the smooth local solution to \eqref{MHD1}-\eqref{MHD4} can be extended beyond the maximal time of existence $T^*$ as defined in the previous section, thereby contradicting the maximality of $T^*$. The following is the main theorem of this section:

\begin{thm}\label{higher order bounds thm}
Assume that the hypotheses and notations in  {\rm Theorem~\ref{a priori bound thm}} are in force. Given $M_0>0$ and $\tilde\rho>0$, assume further that $(\rho-\tilde\rho,u,B)$ satisfies \eqref{pointwise bound on rho finite time}. If $C>0$ is the constant as obtained in Theorem~\ref{a priori bound thm}, then for each $t\in[0,T^*)$, there exists a positive number $M$ which depends on $C$, $M_0$, $t$, $\tilde\rho$ and the system parameters $P$, $\mu$, $\lambda$, $\nu$ such that
\begin{align}\label{3.1}
\sup_{0\le s\le t}||(\rho-\tilde\rho,u,B)||_{H^3(\R^3)}+\int_0^t||(u,B)(\cdot,s)||^2_{H^4(\R^3)}\le M.
\end{align}
\end{thm}

We give the proof of Theorem~\ref{higher order bounds thm} in a sequence of steps. Most of the details are reminiscent of \cite{suenhoff12} and \cite{suen20b}, hence we omit some of those which are identical to arguments given in \cite{suenhoff12} or \cite{suen20b}. Throughout this section, $M$ denotes a generic constant which depends on $M_0$, $t$, $\tilde\rho$, $P$, $\mu$, $\lambda$, $\nu$, and it may be changed from line to line. 

We first begin with the following estimates on the time integral of the velocity gradient $\nabla u$:

\medskip

\noindent{\bf  Step 1:} {\em The velocity gradient satisfies the following bound}
\begin{align}\label{time integral of nabla u}
\int_0^t||\nabla u(\cdot,s)||_{L^\infty}\le M.
\end{align}
\begin{proof}[Proof of Step 1]
The proof is similar to the one given in \cite{suen20b}, and we only sketch here. The key is to decompose $u$ as $u=u_{F}+u_{P}$, where $u_{F}$, $u_{P}$ satisfy
\begin{align}\label{decomposition of u}
\left\{
 \begin{array}{lr}
(\mu+\lambda)\Delta (u_{F})^{j}=F_{x_j} +(\mu+\lambda)(\omega)^{j,k}_{x_k}\\
(\mu+\lambda)\Delta (u_P)^{j}=(P-P(\tilde{\rho}))_{x_j}.\\
\end{array}
\right.
\end{align}
In view of the decomposition \eqref{decomposition of u}, it suffices to bound the time integral of $\|\nabla u_F\|_{L^\infty}$ and $\|\nabla u_P\|_{L^\infty}$. Using \eqref{L infty bound general}, for $r>3$, we have
\begin{align*}
\int_{0}^{t}||\nabla u_{F}(\cdot,s)||_{\infty}\le C(r)\int_{0}^{t}\left[||\nabla u_{F}(\cdot,s)||_{L^r}+||D_{x}^2 u_{F}(\cdot,s)||_{L^r}\right].
\end{align*}
and with the help of the bound \eqref{bounds on phi 1 and 2}, the right side of the above can be bounded by $M$. On the other hand, to bound the time integral of $\nabla u_P$, by the pointwise bound \eqref{pointwise bound on rho finite time} on $\rho$, one can show that $u_P(\cdot,t)$ is, in fact, {\it log-Lipschitz} with bounded log-Lipschitz seminorm. This is crucial for proving that, the integral curve $x(y,t)$ as defined by 
\begin{align*}
\left\{ \begin{array}
{lr} \dot{x}(t)
=u(x(t),t)\\ x(0)=y,
\end{array} \right.
\end{align*}
is H\"{o}lder-continuous in $y$. Upon integrating the mass equation along integral curves $x(t,y)$ and $x(t,z)$, subtracting and recalling the definition \eqref{definition of F} of $F$, we obtain that
\begin{align}\label{Holder norm estimate of rho}
|\log\rho&(x(t,y),t)-\log\rho(x(t,z),t)|\notag\\
&\le|\log\rho_0(y)-\log\rho_0(z)|
+\int_0^t| P(\rho_0(x(s,y),s))-P(\rho(x(s,z),s)|\notag \\
&\qquad\qquad\qquad\qquad\qquad\qquad\quad+\int_0^t|F(x(s,y),s)-F(x(s,z),s)|.
\end{align}
Since $P$ is increasing, the second term of the above is dissipative and can be dropped out. Moreover, with the help of the estimate \eqref{bound on nabla F} on $F$ and the H\"{o}lder-continuity of $x(y,t)$, the third term can be bounded by $M$. Hence we can conclude from \eqref{Holder norm estimate of rho} that $\rho(\cdot,t)$ is $C^{\beta(t)}$ for some $\beta(t)>0$ with bounded modulus. Finally, with the improved regularity on $\rho(\cdot,t)$, we can now make use of \eqref{decomposition of u}$_2$ again and apply properties of Newtonian potentials to conclude that the $C^{1+\beta(t)}(\R^3)$ norm of $u_P$ remains finite in finite time, thereby giving the required bound on $\int_0^t\|\nabla u_P\|_{L\infty}$.
\end{proof}
\noindent{\bf  Step 2:} {\em We further obtain}
\begin{align}
||D^2_{x}u(\cdot,t)||_{L^2}&\le M\left[||\rho \dot{u}(\cdot,t)||_{L^2}+||\nabla B\cdot B(\cdot,t)||_{L^2}+||\nabla P(\cdot,t)||_{L^2}\right],\label{3.5a}
\end{align}
\begin{align}
||D^3_{x}u(\cdot,t)||_{L^2}&\le M\left[||\nabla\rho\cdot\dot u(\cdot,t)||_{L^2}+||\rho\nabla\dot u(\cdot,t)||_{L^2}+||B\cdot D^2_{x}B(\cdot,t)||_{L^2}\right]\notag\\
&\qquad+M\left[|||\nabla B|^2(\cdot,t)||_{L^2}+||D^2_{x}P(\cdot,t)||_{L^2}\right].\label{3.5b}
\end{align}
\begin{proof}[Proof of Step 2]
These follow immediately from the momentum equation \eqref{MHD2} and the ellipticity of the Lam\'{e} operator $\varepsilon\Delta+(\varepsilon+\lambda)\nabla\divv$; see \cite{swz11} for related discussion.
\end{proof}
\noindent{\bf  Step 3:} {\em The following $H^2$ bound for density hol}
\begin{align}\label{3.5c}
\sup_{0\le s\le t}||(\rho-\tilde\rho)(\cdot,s)||_{H^2}\le M.
\end{align}
\begin{proof}[Proof of Step 3]
We take the spatial gradient of the mass equation \eqref{MHD1}, multiply by $\nabla\rho$ and integrate by parts to obtain
\begin{align}
\frac{\partial}{\partial t}\int_{\R^3}|\nabla\rho|^2\le M\left[\int_{\R^3}|\nabla\rho|^2+\int_{\R^3}|D^2_{x} u|^2\right]\label{3.6}
\end{align}
Thanks to \eqref{3.5a}, we have
\begin{align*}
\int_0^t\int_{\R^3}|D^2_{x} u|^2&\le\int_0^t\int_{\R^3}(|\dot u|^2+|\nabla B\cdot B|^2+|\nabla\rho|^2)\\
&\le M+\int_0^t\int_{\R^3}|\nabla\rho|^2,
\end{align*}
hence by applying the above to \eqref{3.6} and using the bound \eqref{time integral of nabla u} on the time integral of $\|\nabla u\|_{L^\infty}$, we conclude that
\begin{align*}
\sup_{0\le s\le t}||\nabla\rho(\cdot,s)||_{L^2}\le M.
\end{align*}
By repeating the argument, one can prove that $\dis\sup_{0\le s\le t}||D^2_{x}\rho(\cdot,s)||_{L^2}\le M$ and \eqref{3.5c} follows.
\end{proof}
\noindent{\bf  Step 4:} {\em The velocity and magnetic field satisfy}
\begin{align}\label{3.7}
\sup_{0\le s\le t}\left(||u(\cdot,s)||_{H^3}+||B(\cdot,s)||_{H^3}\right)\le M.
\end{align}
\begin{proof}[Proof of Step 4]
Define the forward difference of quotient $D^h_{t}$ by
\begin{align*}
D^h_t(f)(t)=(f(t+h)-f(t))h^{-1}
\end{align*}
and let $E^j=D^h_t(u^j)+u\cdot\nabla u^j$. By applying $E^j$ on the momentum equation \eqref{MHD2} and differentiating, it gives
\begin{align*}
&\int_{\R^3}\rho|E_{x_j}|^2+\int_0^t\int_{\R^3}\left(|\nabla E_{x_j}|^2+|D^h_t(\divv(u_{x_j})+u\cdot\nabla(\divv(u_{x_j})|^2\right)\\
&\qquad\qquad\qquad\le M+\int_0^t\int_{\R^3}|\nabla E|^2+\mathcal O(h),
\end{align*}
where $\mathcal O(h)\rightarrow 0$ as $h\rightarrow0$. Therefore by choosing $h\rightarrow0$, we conclude
\begin{align*}
\sup_{0\le s\le t}||\nabla\dot u(\cdot,s)||_{L^2}+\int_0^t\int_{\R^3}|D^2_x\dot u|^2\le M,
\end{align*}
and the bound for $\nabla B_t$ can be derived in a similar way.
\end{proof}
\noindent{\bf  Step 5:} {\em Finally we have the following boun}
\begin{align}
&\int_0^t\int_{\R^3}(|D^4_{x}u|^2+|D^4_{x}B|^2)\le M\left[1+\int_0^t\int_{\R^3}|D^3_{x}\rho|^2\right],\label{3.8a}\\
&\sup_{0\le s\le t}\left(||D^3_{x}\rho(\cdot,s)||_{L^2}+||D^3_{x}B(\cdot,s)||_{L^2}\right)+\int_0^t\int_{\R^3}|D^4_{x}u|^2\le M.\label{3.8b}
\end{align}
\begin{proof}[Proof of Step 5]
To prove \eqref{3.8a}, we differentiate \eqref{MHD2} and \eqref{MHD3} twice with respect to space, express the fourth derivatives of $u$ and $B$ in the terms second derivatives of $\dot u$, $B_t$, $\nabla\rho$ and lower order terms, and apply the bounds in \eqref{pointwise bound on rho finite time} and \eqref{3.7}.

On the other hand, to prove \eqref{3.8b}, we apply two space derivatives and one spatial difference operator $D^h_{x_j}$ defined by
\begin{align*}
D^h_{x_j}(f)(t)=(f(x+he_j)-f(x))h^{-1}
\end{align*}
such that
\begin{align*}
\int_{\R^3}|D^h_{x_j} D_{x_i} D_{x_k}\rho|^2&\le M+\int_0^t\int_{\R^3}(|D^4_x u|^2+|D^h_{x_j}D_{x_i}D_{x_k}\rho|^2)\\
&\le M+\int_0^t\int_{\R^3}|D^3_{x}\rho|^2.
\end{align*}
Taking $h\rightarrow0$ and applying Gronwall's inequality, we obtain the required bound for the term $\dis||D^3_{x}\rho(\cdot,s)||_{L^2}$. 
\end{proof}

\begin{proof}[Proof of Theorem~\ref{main thm}]
Using Theorem~\ref{higher order bounds thm}, we can apply an open-closed argument on the time interval which is identical to the one given in Hoff and Suen \cite{suenhoff12} pp. 31 to extend the local solution $(\rho-\tilde\rho,u,B)$ beyond $T^*$, which contradicts the maximality of $T^*$. Therefore the assumption \eqref{pointwise bound on rho finite time} does not hold and this completes the proof of Theorem~\ref{main thm}.
\end{proof}



\bibliographystyle{amsalpha}

\bibliography{References_for_blow_up_MHD_general}

\end{document}